\documentclass[12pt]{extarticle}
\usepackage{amsmath, amsthm, amssymb, color}
\usepackage[colorlinks=true,linkcolor=blue,urlcolor=blue]{hyperref}
\usepackage{graphicx}
\usepackage{caption}
\usepackage{titlesec}
\titleformat{\section}[hang]
  {\normalfont\bfseries}{\thesection.}{1em}{}
\titleformat{\subsection}
  {\normalfont\normalsize\itshape}{\thesubsection.}{1em}{}
\titleformat{\subsubsection}
  {\normalfont\normalsize\itshape}{\thesubsubsection.}{1em}{}
\usepackage{mathtools}
\usepackage{cleveref}
\usepackage{enumerate}
\usepackage{verbatim}
\usepackage{tikz,tikz-cd,tikz-3dplot}
\usetikzlibrary{graphs,graphs.standard}
\usepackage{amssymb}
\usetikzlibrary{matrix}
\usetikzlibrary{arrows}
\usepackage{algorithm}
\usepackage[noend]{algpseudocode}
\usepackage{caption}
\usepackage[normalem]{ulem}
\usepackage{subcaption}
\oddsidemargin \evensidemargin
\usepackage{makecell}
\usepackage{array}
\usepackage{ytableau}

\newtheorem{theorem}{Theorem}
\newtheorem{proposition}[theorem]{Proposition}
\newtheorem{lemma}[theorem]{Lemma}
\newtheorem{corollary}[theorem]{Corollary}

\theoremstyle{definition}

\newtheorem{remark}[theorem]{Remark}
\newtheorem{conjecture}[theorem]{Conjecture}

\newtheorem{example}[theorem]{Example}

\newcommand{\PP}{\mathbb{P}} 
\newcommand{\RR}{\mathbb{R}}

\newcommand{\CC}{\mathbb{C} }
\newcommand{\ZZ}{\mathbb{Z}}
\newcommand{\A}{\mathbb{A}}

\title{\bf Coupled Cluster Degree of the~Grassmannian}

\author{Viktoriia Borovik,
Bernd Sturmfels and Svala Sverrisdóttir}

\date{}
\begin{document}
\maketitle

\begin{abstract}
\noindent    
We determine the~number of complex solutions to a~nonlinear eigenvalue problem
 on~the~Grassmannian  in~its Pl\"ucker embedding.   This is motivated by
 quantum chemistry, where it represents the~truncation to single electrons in~coupled cluster theory. 
 We prove the~formula for the~Grassmannian of lines which was conjectured in earlier work with
 Fabian Faulstich. This rests on the geometry of  the
 graph of a birational parametrization of the Grassmannian. 
 We present  a squarefree Gr\"obner basis for this graph,
 and we develop connections to
   toric degenerations from representation theory.
        \end{abstract}

\section{Introduction}

\noindent
The electronic Schrödinger equations describe the behavior of electrons within molecules and crystals.
This is of great interest in the field of quantum chemistry. The equations represent an eigenfunction problem for a differential operator called the Hamiltonian. Using discretization we can turn the problem into a 
finite-dimensional eigenvalue problem for the Hamiltonian matrix. Due~to the high dimensionality of this
eigenvalue problem, efficient and tractable numerical
schemes are essential for approximating the equations. Coupled cluster theory,
which is considered the gold standard in quantum chemistry, approximates the equations by a hierarchy of polynomial systems at various
levels of truncation \cite{FO, FSS}. 

In this paper we focus on the CCS (Coupled Cluster Single) model, allowing only single level electronic excitations. The coupled cluster equations
for the CCS model take the~form
\begin{equation}
\label{eq:CCeqns}
(H \psi)_* \,= \,\lambda \,\psi_* \quad {\rm for} \quad \psi \in {\rm Gr}(d,n). 
\end{equation}
Here,~${\rm Gr}(d,n)$ is the~Grassmannian~in its
Pl\"ucker embedding in~$\PP^{\binom{n}{d}-1}$,
and~$\psi$ is the~column vector of  Pl\"ucker coordinates~$\psi_{i_1 i_2 \cdots i_d}$.
These are indexed by~$ 1 \leq i_1 < i_2 < \cdots < i_d \leq n$.
 We use the~letter~$\psi$ because it represents quantum states 
  for~$d$ electrons in~$n$ spin-orbitals.
The star~$*$ denotes the~truncation to the~$d(n-d)+1$ coordinates
which satisfy~$i_{d-1} \leq d$. Our system
 (\ref{eq:CCeqns}) is a~truncation of the~eigenvalue problem for 
 the~symmetric~$\binom{n}{d} \times \binom{n}{d}$ matrix~$H$,
  which is the~Hamiltonian.
The number of constraints on~$\psi$ imposed by (\ref{eq:CCeqns})
 is equal to  the~dimension~$d(n-d)$ of~${\rm Gr}(d,n)$.
The set of complex solutions  is finite for generic~$H$. 
The cardinality of this set  is the
coupled cluster degree (CC degree). It
depends only on~$d$ and~$n$.

The Grassmannian is the~closure in~$\PP^{\binom{n}{d}-1}$ of the~image of the~parametrization map
$$  \gamma: X = [x_{ij}] \,\mapsto \,
\hbox{all~$\binom{n}{d}$ maximal minors of} \,\,\, [\,{\rm I}_d \, X] \,= 
\begin{footnotesize} \begin{bmatrix}
1 & 0 & \cdots & 0  & x_{1,d+1} & x_{1,d+2} & \cdots & x_{1,n} \\
0 & 1 & \cdots & 0  & x_{2,d+1} & x_{2,d+2} & \cdots & x_{2,n} \\
\vdots & \vdots & \ddots & \vdots & \vdots & \vdots & \ddots & \vdots \\
0 & 0 & \cdots & 1 &  x_{d,d+1} & x_{d,d+2} & \cdots & x_{d,n} \\
\end{bmatrix}\! .
\end{footnotesize}
$$
This maps the~$d \times (n-d)$ matrix~$X$ to all its minors. We view this as a rational map:
$$ \gamma \,:\, \PP^{d(n-d)} \,\dashrightarrow \,{\rm Gr}(d,n)\, \subset \,\PP^{\binom{n}{d}-1}.~$$
In fact, we define the map $\gamma$ on an affine subset $\mathbb{C}^{d(n-d)}$ and extend it to the whole $\mathbb{P}^{d(n-d)}$ by multiplying the matrix~$\mathrm{I}_d$ by an additional variable $t$.
The closure of the graph of the map~$\gamma$ is a~subvariety~$\mathcal{G}(d,n)$ of
dimension~$d(n-d)$ in
the product space $\PP^{d(n-d)} \times \PP^{\binom{n}{d}-1}$.
Its cone in the affine space $\A^{d(n-d) + 1} \times \A^{{n \choose d}}$
has dimension $d(n-d)+2$. In the following theorem we refer to the
degree of the associated projective variety of dimension $d(n-d)+1$.

\begin{theorem} \label{thm:graphmap}
The CC degree of the~Grassmannian~${\rm Gr}(d,n)$ is the total degree of the
graph~$\,\mathcal{G}(d,n)$, here thought of as the projectivization of its affine cone
in $\A^{d(n-d) + 1} \times \A^{{n \choose d}}$.
\end{theorem}

This result will be explained and proved in~Section~\ref{sec2}.
Our approach is to argue that the~following
bilinear equations on~$ \PP^{d(n-d)} \times \PP^{\binom{n}{d}-1}$ are
equivalent to the~CC equations (\ref{eq:CCeqns}):
\begin{equation}
\label{eq:linsec}
 (H \psi)_* \,= \,\lambda\, \xi   \quad {\rm for} \quad (\xi,\psi) \in \mathcal{G}(d,n).
 \end{equation}
 The coordinates $\xi_{i_1i_2\cdots i_d}$ on $\PP^{d(n-d)}$ are indexed by $ 1 \leq i_1 < i_2 < \cdots < i_d \leq n$ with $i_{d - 1} \le d$. They correspond to the variables $x_{ij}$ such that $i = [d] \backslash \{i_1,\ldots, i_{d - 1}\}$ and $j = i_d$ when $i_d > d$. The coordinate $\xi_{[d]}$ corresponds to the additional variable $t$. The~solutions to the equation~(\ref{eq:linsec}) are counted by the~degree of the~graph~$\mathcal{G}(d,n)$.
In Section~\ref{sec3} we are focusing on~the~case~$d=2$. We  prove
the following formula, which appeared in~\cite[Conjecture~5.5]{FSS}.

\begin{theorem} \label{thm:conj55}
The CC degree of the~Grassmannian~${\rm Gr}(2,n)$ 
is equal to~$\, \frac{2}{n} \binom{2n-2}{n-1} - 1$.
\end{theorem}

The proof of Theorem~\ref{thm:conj55} rests on a certain~Gr\"obner basis. Additionally we 
describe a Khovanskii basis that gives a toric degeneration of the graph. Both flavors of bases
are well-studied for Grassmannians, and we
lift them from~${\rm Gr}(2,n)$ to the~graph~$\mathcal{G}(2,n)$.
In Section~\ref{sec4} the~same approach is developed for~$d \geq 3$.
Here, it is essential to make judicious choices
of monomial orders, building on~recent advances 
on toric degenerations in
\cite{MohClarke, FFFM, KM, MAKHLIN2022105541, IgMal22}.

\smallskip

We conclude the~introduction with an example that illustrates some  of the~key ideas.

\begin{example}[$d=2,n=6$]
The prime ideal of the~Grassmannian~${\rm Gr}(2,6) \subset \PP^{14}$
is generated by the~$15$ Pl\"ucker quadrics 
$\,\psi_{il} \psi_{jk} - \psi_{ik} \psi_{jl}+\psi_{ij} \psi_{kl}    \,$
for~$1 \leq i < j < k < l \leq 6$. 
   For the graph~$\mathcal{G}(2,6) \subset \PP^8 \times \PP^{14}$, we augment the~Pl\"ucker ideal by the
$2 \times 2$ minors of the~$2 \times 9$~matrix 
\begin{equation}\label{rankonematrix}
\begin{bmatrix}
\psi_{12} & \psi_{13} & \psi_{14} & \psi_{15} & \psi_{16} & \psi_{23} & \psi_{24} & \psi_{25} & \psi_{26} \\
\xi_{12} & \xi_{13} & \xi_{14} & \xi_{15} & \xi_{16} & \xi_{23} & \xi_{24} & \xi_{25} & \xi_{26} \\
\end{bmatrix}.
\end{equation}
The bihomogeneous prime ideal of~$\mathcal{G}(2,6)$  is obtained by
saturating with respect to $\psi_{12}$.
This ideal is minimally generated by
the~$15$ Pl\"ucker quadrics of bidegree~$(0,2)$, and
$50$ quadrics of bidegree~$(1,1)$.
In addition to the~$36$ maximal minors of  (\ref{rankonematrix}), 
there are~$14$ polarized Pl\"ucker relations, like
$\psi_{14}\xi_{23}-\psi_{13}\xi_{24}+\psi_{34}\xi_{12}$. 
The entries of the~$ 2\times 4$ matrix~$X$ in~the~definition of~$\gamma$
serve as local coordinates. In these coordinates, the
matrix (\ref{rankonematrix}) is seen to have rank one:
$$ \begin{bmatrix} t\\ s \end{bmatrix} \cdot
\begin{bmatrix}
\,1 \!&\! x_{23} \!&\! x_{24} \!&\! x_{25} \!&\! x_{26} \! & \! -x_{13} \!&\! -x_{14} \!&\! -x_{15} \!&\! -x_{16}  \\
 \end{bmatrix}.~$$
By applying the~command {\tt multidegree} in~{\tt Macaulay2} to our bihomogeneus ideal,
 we find that the~class of  the~graph~$\mathcal{G}(2,6)$ in
the cohomology ring~$ H^*( \PP^8 \times \PP^{14}) = \ZZ[s,t]/\langle s^9, t^{15} \rangle$ is
$$ s^8 t^6 +  2s^7 t^7 + 4s^6 t^8 + 8s^5 t^9 + 12s^4 t^{10} + 14s^3 t^{11} + 14s^2 t^{12} + 14s t^{13} + 14 t^{14} .~$$
The total degree of the graph $\mathcal{G}(2,6)$ is the sum of the
coefficients of this binary form. This sum equals~$83$, so it agrees with the
  CC degree of~${\rm Gr}(2,6)$, by
\cite[Example 5.4]{FSS}.
The proof of Theorem~\ref{thm:conj55} in~Section~\ref{sec3} rests on~the~fact that the~$65$ 
ideal generators of~$\mathcal{G}(2,6)$ form~a Gr\"obner basis.
A combinatorial study of its squarefree leading monomials reveals the
Catalan number~$  42 = {\rm degree}({\rm Gr}(2,7)) $, which  is the~key ingredient for
${\rm degree}(\mathcal{G}(2,6)) =  2 \cdot  42-1$.
\hfill~$\diamond$ 
\end{example}

Using the~combinatorial and computational techniques
to be developed in~Section~\ref{sec4}, it is now possible to compute some
new CC degrees that go beyond those reported in~\cite{FSS}.

\begin{example} \label{ex:bigCC}
Here are three new values for the~census of CC degrees of Grassmannians:
$$ {\rm degree}( \mathcal{G}(3,9) ) = 574507, \;\; {\rm degree}( \mathcal{G}(3,10) ) = 9239646 \;
\text{  and  } \;
{\rm degree}( \mathcal{G}(4,9) ) = 10907231.~$$
These CC degrees are at most around $10^7$, which means that all complex solutions of
the corresponding CC equations can be found numerically with {\tt HomotopyContinuation.jl}~\cite{HomotopyContinuation}.
\end{example}

The code that verifies Example \ref{ex:bigCC}, along with other supplementary materials
for this article,  is made available at the repository website {\tt MathRepo}~\cite{mathrepo} of MPI-MiS
via the link 
\begin{equation}
\label{eq:mathrepo} \hbox{\url{https://mathrepo.mis.mpg.de/CCdegreeGrassmannian/}.} 
\end{equation}

\section{Intersecting the~graph}
\label{sec2}

The objective of this section is to prove Theorem~\ref{thm:graphmap}.
We need the~following general lemma.
Let~$V$ be any irreducible subvariety of dimension~$d$
in~$\PP^m \times \PP^n$. Fix the homogeneous coordinates
$\mathbf{x} = [x_0:\dots:x_m]$ and~$\mathbf{y} = [y_0:\dots:y_n]$ 
on the~two factors. For~$i=0,1,\ldots,d$ we consider 
generic linear forms~$\ell_i(\mathbf{x})$ 
 and~$m_i(\mathbf{y})$, defining hyperplanes in~$\PP^m$ and
~$\PP^n$ respectively.

\begin{lemma} \label{lem:degree}  The total~degree of the~variety~$V$
equals the~number of points~$(\mathbf{x},\mathbf{y})$ on~$V$ such~that
the vectors
$(\ell_0(\mathbf{x}),\ell_1(\mathbf{x}),\ldots,\ell_d(\mathbf{x}))$ and 
$(m_0(\mathbf{y}),m_1(\mathbf{y}),\ldots,m_d(\mathbf{y}))$
are linearly dependent.
\end{lemma}

\begin{proof}
Recall that the~multidegree of~$V$ is the~binary form that represents the~cohomology class~$[V] \in H^*(\PP^m \times \PP^n) = \ZZ[s,t]/ \langle s^{m+1}, t^{n+1} \rangle~$. The~degree of~$V$ is the~sum of the~coefficients of the~multidegree~$[V]$, 
see e.g.~\cite{VANDERWAERDEN1978303}.
We consider the~$2 \times (d+1)$ matrix
whose rows are the~two vectors introduced above.
Its~$2 \times 2$ minors are the
bilinear forms~$\ell_i(\mathbf{x}) \cdot m_j(\mathbf{y}) - \ell_j(\mathbf{x})\cdot  m_i(\mathbf{y})$ 
for~$ 0 \leq i < j \leq d$. We need to count
their zeros on~$V$.
To this end, we introduce a~small real parameter~$\varepsilon >0$,
and we replace~$\ell_i(\mathbf{x})$ by
~$\varepsilon^i \cdot \ell_i (\mathbf{x})$ for all~$i$.
Then our equations become
$$\,\ell_i(\mathbf{x})\cdot  m_j(\mathbf{y})  \,=\, \varepsilon^{j-i} \cdot \ell_j(\mathbf{x}) \cdot m_i(\mathbf{y}) \,
\quad \hbox{ for~$ \,\,0 \leq i < j \leq d$}.~$$

When $\varepsilon$ approaches  zero, we obtain the~equations
$\, \ell_i(\mathbf{x})\cdot  m_j(\mathbf{y})  = 0 \,$ for~$\,\,0 \leq i < j \leq d$.
This system decomposes into~$d+1$ systems of~$d$ linear equations which are indexed by~$ k = 0,1,\ldots,d$.
Each of them amounts to intersecting~$V$ with a~product of subspaces
in~$\PP^m \times \PP^n$:
\begin{equation}
\label{eq:lmsystem}
 \ell_0 (\mathbf{x}) = \cdots = \ell_{k-1}(\mathbf{x}) \, = \, m_{k+1}(\mathbf{y}) = \cdots = m_d(\mathbf{y}) \,=\, 0 
 \quad \hbox{for~$k = 0,1,\ldots,d$.}
 \end{equation}
 This is a system of $d$ linear equations.
 Since the~linear forms~$\ell_i$ and~$m_j$ are generic, the system~\eqref{eq:lmsystem} has
  finitely many solutions~$(\mathbf{x},\mathbf{y})$ on the
  $d$-dimensional variety $V$, and each solution 
is nondegenerate. 
By the~Implicit Function Theorem, each solution extends from~${\varepsilon = 0}$ to a~small positive value~$\varepsilon_0 > 0$, and these are all
 solutions of the~original system for~$\varepsilon_0$.
 The~number of solutions to (\ref{eq:lmsystem}) is
a multidegree of~$V$, namely it is the~coefficient of~$s^{m-k} t^{n-d+k}$ in~the~binary form~$[V]$. By summing over all~$k$,
we obtain the total~degree of~$V$, as desired.
\end{proof}

\begin{proof}[Proof of Theorem~\ref{thm:graphmap}]
We shall apply Lemma~\ref{lem:degree} to the variety
$V := \mathcal{G}(d,n)$, which is the~graph of the~parametrization~$\gamma$ of
the Grassmannian~${\rm Gr}(d,n)$. This
$V$ lives in~$\PP^{d(n-d)} \times \PP^{\binom{n}{d}-1}$ and it
has dimension~$d(n-d)$. 
We denoted the~coordinates by~${\xi = \bf x}$ and
${\psi = \bf y}$, in~order to be consistent with \cite{FO, FSS}.
After eliminating~$\lambda$ from (\ref{eq:CCeqns}), 
the CC equations on~the~Grassmannian~${\rm Gr}(d,n)$ are the
$2 \times 2$ minors of the~two-column matrix
$\bigl[ (H \psi)_* \, , \, \psi_* \bigr]$.
The right column is equal to the~vector~$\xi$ on
the graph~$\mathcal{G}(d,n)$. We can thus write the
CC equations on~$\mathcal{G}(d,n)$ as the~$2 \times 2$ minors of
$\bigl[ (H \psi)_* \, , \, \xi \bigr]$.
This matrix has
two columns and~$d(n-d)+1$ rows.

Hence the~system of CC equations is precisely as in~Lemma~\ref{lem:degree},
where~$\ell_i(\xi)$ is the~$i$th coordinate of~$ \xi$ and
$m_j(\psi)$ denotes the~linear form given by the~$j$th row in~the~Hamiltonian~$H$.
The coordinate functions~$\ell_i(\xi) = \xi_i$ are not generic,
but we can replace them with generic linear forms by a~linear
change of coordinates on~$\PP^{d(n-d)}$, given by an invertible matrix~$M$ of size~$d(n-d)+1$.
This coordinate change can be compensated by multiplying the~Hamiltonian 
$H$ on~the~left by the~block matrix~$M^{-1} \oplus I_{\binom{n}{d}-d(n-d)-1}$,
which is square of size~$\binom{n}{d}$.
Since the Hamiltonian matrix~$H$ is generic, the resulting linear forms exhibit the~generic behavior
that is needed for  Lemma~\ref{lem:degree} to be applicable.
We conclude that the~number of complex solutions to the
CC equations (\ref{eq:CCeqns})
 is equal to the total degree of the variety~$V = \mathcal{G}(d,n)$.
\end{proof}

\begin{remark}
The definition of the~CC degree in~\cite{FSS}
counts the~number of solutions to the~CC equations with multiplicities. 
It still needs to be shown that there exists a~Hamiltonian~$H$ for which all solutions
are nondegenerate. Ideally, we want them to be all real. 
If we relax the requirement that
the square matrix $H$ should be symmetric, then this would follow from the
squarefree Gr\"obner basis promised in Conjecture \ref{thm:KhovanskiiPBW}. Namely, 
representing the monomial order by a weight vector,
the Gr\"obner degeneration corresponds to scaling the unknowns:
$$ \psi \rightarrow D_1(\varepsilon) \cdot \psi \quad {\rm and} \quad \xi \rightarrow D_2(\varepsilon) \cdot \xi .$$
Here $D_1$ and $D_2$ are diagonal matrices whose entries are different powers of the parameter $\varepsilon$. 
Our polynomial system on $\mathcal{G}(d,n)$ is now given by the $2 \times 2$ minors of
$\bigl[ (H D_1(\varepsilon) \psi)_* \, , \,  D_2(\varepsilon) \xi \bigr]$.

For $\varepsilon =0$, we are solving linear equations on an arrangement
of reduced linear subspaces, because the initial monomial ideal is radical.
Here, all  solutions are real and nondegenerate.
By the Implicit Function Theorem,  each solution extends
to a small $\varepsilon_0 > 0$. We now define  
$$D_3(\varepsilon_0) :=  D_2(\varepsilon_0)^{-1} \oplus I_{\binom{n}{d}-d(n-d)-1}\quad  \text{and} \quad H(\varepsilon_0) := D_3(\varepsilon_0) \cdot H \cdot D_1(\varepsilon_0).$$
With these matrices, 
our polynomial system is given by the $2 \times 2$ minors of~$\bigl[ (H(\varepsilon_0) \psi)_* \, , \,  \xi \bigr]$. Thus, using $H(\varepsilon_0)$ for the 
Hamiltonian, all solutions to the CC equations are real and nondegenerate.
At present we do not know how to replace $H(\varepsilon_0)$ by a symmetric matrix.
\end{remark}

\section{Catalan numbers}
\label{sec3}

The degree of the~Grassmannian~${\rm Gr}(2,n)$ is the Catalan number
$C_{n-2} = \frac{1}{n-1}\binom{2n-4}{n-2}$.
Theorem~\ref{thm:conj55} states that
the graph~$\mathcal{G}(2,n)$ has degree~$2 C_{n-1}-1$.
In this section we shall prove~this.

For the~proof it is convenient
to permute the~columns in~the~parametrization of 
${\rm Gr}(2,n)$. Namely, we redefine~$\gamma$
to be the~map that evaluates  the~$2 \times 2$ minors of
\begin{equation}
\label{eq:matrix2n}
  \begin{bmatrix}
1 & x_{1,2} & x_{1,3} & \dots & x_{1,n-1} & 0 \\
0 & x_{2,2} & x_{2,3} & \dots & x_{2,n-1} & 1 
\end{bmatrix}\! .
\end{equation}
As before,~$\mathcal{G}(2,n)$ is the~closure of the
graph of~$\gamma$. This is a~subvariety of dimension $2n-4$ in
$\PP^{2n-4} \times \PP^{\binom{n}{2}-1}$.
The prime ideal of~$\mathcal{G}(2,n)$ is the~kernel of the~ring homomorphism:
\begin{equation}
\label{eq:ringhomo2a} \begin{array}{rclccc}
\CC[\xi,\psi] & \rightarrow& \CC[s,t, \bf x],   &\xi_{1n}  \mapsto s, & \xi_{1i} \mapsto s  x_{2i}, 
& \xi_{jn}  \mapsto s  x_{1j},  \\ 
\psi_{ij}  & \mapsto &  t (x_{1i} x_{2j} - x_{1j} x_{2i}), &\psi_{1n}  \mapsto t, &\psi_{1i}  \mapsto t  x_{2i}, 
&\psi_{jn}  \mapsto t  x_{1j},
\end{array} \quad
1 < i < j < n. 
\end{equation}
The polynomial ring~$\CC[\xi,\psi]~$ has~$2n-3 + \binom{n}{2}$ variables.
We order these variables as follows:
\begin{equation}
\label{eq:termorder2n} \begin{matrix}
\xi_{12} < \psi_{12}\, < \,\xi_{13} < \psi_{13} \,<  \, \cdots \,<\,
\xi_{1n} < \psi_{1n} \,<\, \psi_{23} < \psi_{24} \,< \,\cdots \,<\, \psi_{2,n-1} \,<\, \xi_{2n} < \psi_{2n} \\
<\, \psi_{34} \,< \,\psi_{35} \,<\,
 \cdots \,<\, \psi_{n-2,n-1} \,<\, \xi_{n-2,n} < \psi_{n-2,n}  \,<\, \xi_{n-1,n}  < \psi_{n-1,n}.
\end{matrix}
\end{equation}
In words, we order the~variables~$\psi_{ij}$ by sorting their index pairs
lexicographically, and we then  insert each variable~$\xi_{1i}$ or~$\xi_{jn}$
right before the~$\psi$-variable with the~same index~pair.
We fix the~reverse lexicographic monomial order on~$\CC[\xi,\psi]$ that is induced by
this variable order.

\begin{lemma}\label{GrBasisG(2,n)}
The following
$(n-1)(n-2)(n^2+5n+12)/24$
 quadrics minimally  generate the~ideal of~$\,\mathcal{G}(2, n)$, and they form
 the~reduced Gr\"obner basis for 
 the~monomial order described above:
$$ \begin{array}{clll}
 \binom{n}{4} & {\rm trinomials} & 
    \underline{
    \psi_{il} \psi_{jk}} - \psi_{ik} \psi_{jl} + \psi_{ij} \psi_{kl} 
    &   {\rm for} \,\, 1 \leq i \! < \!  j \! <\! k \! <\! l \leq n, \smallskip \\ 
 \binom{n-1}{2} & {\rm binomials} & 
 \underline{ \psi_{in} \xi_{jn}} - \psi_{jn} \xi_{in} 
 & {\rm for} \,\, 1 \leq i < j < n, \smallskip \\
 (n-2)^2 & {\rm binomials} & 
 \underline{ \psi_{1j}\xi_{kn}} -  \psi_{kn} \xi_{1j}& 
 {\rm for} \,\,1 <\, j\,, \, k \, < n, \smallskip \\
 \binom{n-1}{2} & {\rm binomials} & 
 \underline{ \psi_{1k}\xi_{1l}} -  \psi_{1l}\xi_{1k} &
   {\rm for} \,\, 1 < k < l \leq n,  \smallskip \\
    \binom{n-2}{3} & {\rm trinomials} & 
\underline{ 
\xi_{1l} \psi_{jk}} - \xi_{1k} \psi_{jl} + \xi_{1j} \psi_{kl} &
{\rm for} \,\, 1 < j<k<l < n, \\
\binom{n-1}{3} & {\rm trinomials} & 
\underline{
\xi_{in} \psi_{jk}} - \xi_{jn} \psi_{ik} + \xi_{kn} \psi_{ij} & 
{\rm for} \,\, 1 \leq i < j < k <n.
    \end{array}
$$
\end{lemma}

\begin{proof}
We first check that the~six classes of quadrics
vanish on~$\mathcal{G}(2,n)$. The~$\binom{n}{4}$
trinomials in~$\psi$ are the~Pl\"ucker relations which generate
the ideal of~${\rm Gr}(2,n)$.
Next come three groups of binomials, for a~total of
$\binom{n-1}{2} + (n-2)^2 + \binom{n-1}{2} = \binom{2n-3}{2}$.
These are the~$2 \times 2$ minors of the~matrix 
whose two rows are the~$\psi$ variables and the~matching~$\xi$ variables.
This matrix is displayed in~(\ref{rankonematrix}) for~$n=6$.
It has rank one on~the~graph~$\mathcal{G}(2,n)$, as we can see from (\ref{eq:ringhomo2a}).
The last two groups are~$\binom{n-2}{3} + \binom{n-1}{3} = 
\binom{n}{4} - \binom{n-2}{4}$
 trinomials that are bilinear in~$\xi$ and~$\psi$.
These are obtained from the~Pl\"ucker relations by
applying the~$2 \times 2$ minors of the~previous matrix.
They are also mapped to zero under the~ring map (\ref{eq:ringhomo2a}) that describes
~$\mathcal{G}(2,n)$.

We next check that the~underlined monomials are the~leading monomials
with respect to the~reverse lexicographic monomial order 
that is induced by the~variable order (\ref{eq:termorder2n}).
Let~$M$ be the~ideal generated by these monomials. Since
the monomials are squarefree,~$M$ is a~radical ideal.
We have argued that~$M \subseteq \mathrm{in}(\mathcal{G}(2, n))$,
and we now need to show that equality holds.

Consider any monomial in~the~initial ideal of~$\mathcal{G}(2, n)$.
It is the~leading monomial~${\rm in}(f)$ of some
 bihomogeneous polynomial~$f(\xi,\psi)$ that vanishes
on~$\mathcal{G}(2,n)$.
We write~$f =
    \sum c_{\mathbf a, \mathbf b} \mathbf \xi^{\mathbf a} \mathbf \psi^{\mathbf b}$,
    so our polynomial has bidegree~$(|\mathbf a|, |\mathbf b|) \in \ZZ_{\geq 0}^2$. Note that~$|\mathbf b| > 0$ because the~images of the~$\xi$-variables under  the~ring map
    (\ref{eq:ringhomo2a}) are   algebraically independent.

First assume~$|\mathbf a| = 0$. Then~$f=\sum c_{\mathbf b} \mathbf \psi^{\mathbf b}$ lies in~the~ideal
of the~Grassmannian~${\rm Gr}(2,n)$. The~quadrics~$ \underline{
    \psi_{il} \psi_{jk}} - \psi_{ik} \psi_{jl} + \psi_{ij} \psi_{kl}~$
  are known to form a~Gr\"obner basis  for this ideal
  with respect to our monomial order.
  See e.g.~\cite[Theorem 5.8]{INLA} or \cite[Theorem 14.6]{CCA}.
     Therefore, the leading monomial~$\mathrm{in}(f)$ is a multiple of~$\psi_{il}\psi_{jk}$ for some~$1\leq i < j < k < l \leq n$.

Next suppose~$|\mathbf a| > 0$. 
We consider the~polynomial~$\sum c_{\mathbf{a}, \mathbf{b}} \mathbf \psi^{\mathbf{a} + \mathbf{b}}$
that is obtained from~$f$ by replacing each~$\xi_{ij}$ with the~corresponding~$\psi_{ij}$.
We see from    (\ref{eq:ringhomo2a}) that this polynomial
     vanishes on~$ \mathcal{G}(2, n)$ and hence on~${\rm Gr}(2,n)$.
We write the~leading monomial of~$f$ as follows:
$$ {\rm in}(f) \,\,= \,\,
\xi_{\sigma^{(1)}}^{a_1}\cdots \xi_{\sigma^{(h)}}^{a_h} \cdot \psi_{\tau^{(1)}}^{b_1}\cdots \psi_{\tau^{(r)}}^{b_r}.~$$
Here~$\sigma^{(i)}$ and~$\tau^{(j)}$ are index pairs.    
A key property  of our monomial order is that it respects the~substitution
$\xi \mapsto \psi$. Using this, we
    now examine the~leading monomial of~$\sum c_{\mathbf{a}, \mathbf{b}} \mathbf \psi^{\mathbf{a} + \mathbf{b}}~$.

First assume that the~leading monomial of~$\sum c_{\mathbf{a}, \mathbf{b}} \mathbf \psi^{\mathbf{a} + \mathbf{b}}~$ differs from the monomial
$$\psi_{\sigma^{(1)}}^{a_1}\cdots \psi_{\sigma^{(h)}}^{a_h} \cdot \psi_{\tau^{(1)}}^{b_1}\cdots \psi_{\tau^{(r)}}^{b_r}.$$ 
Then $f$ contains another  monomial~$m$ which cancels~${\rm in}(f)$ after the substitution~$\xi \mapsto \psi$.
Let~$\rho$ be the smallest index pair in the list~$ \sigma^{(1)},\ldots ,\sigma^{(h)}, \, \tau^{(1)},\ldots ,\tau^{(r)}$. 
From the cancellation and the definition of the reverse lexicographic order, we see that
$\rho = \tau^{(j)}$ for some~$j$. Hence~${\rm in}(f)$ contains~$\psi_\rho$ and
$m$ contains~$\xi_\rho$. In particular,~$\rho \cap \{1,n\} \not=  \varnothing$.
 This implies that~${\rm in}(f)$ is divisible
          by one of the~underlined leading monomials of the~binomials in~our list.

Next we suppose that $\psi_{\sigma^{(1)}}^{a_1}\cdots \psi_{\sigma^{(h)}}^{a_h} \cdot \psi_{\tau^{(1)}}^{b_1}\cdots \psi_{\tau^{(r)}}^{b_r}$ is the~leading monomial of~$\,\sum c_{\mathbf{a}, \mathbf{b}} \psi^{\mathbf{a} + \mathbf{b}}$,
which is in~the~ideal of~${\rm Gr}(2,n)$. This leading monomial is divisible
by~$\psi_{il} \psi_{jk}$ for some indices~$1\leq i < j < k < l \leq n$.
Since~$\{j,k\} \cap \{1, n\} = \varnothing$,
the corresponding factor  of~${\rm in}(f)$ is either
$\psi_{il} \psi_{jk}$ or~$\xi_{il} \psi_{jk}$.
Both of these occur among the
 underlined monomials in~our list.
 We conclude that~${\rm in}(f)$ is in~the~ideal generated by the
 underlined leading monomials. 
 
 We have proved that our list of quadratic binomials and trinomials is a~Gr\"obner basis
 for~$\mathcal{G}(2,n)$.
 To see that it is a~reduced Gr\"obner basis, one checks no trailing term
 occurs among the~leading monomials. Since all polynomials have the~same
 degree two, it follows that they minimally generate the~ideal.
 This completes the~proof of Lemma~\ref{GrBasisG(2,n)}.
\end{proof}

Let~$M$ be the~ideal generated by the~underlined monomials in~Lemma~\ref{GrBasisG(2,n)}.
We have shown that~$M$ is the~initial ideal of the~graph~$\mathcal{G}(2,n)$
with respect to the~monomial order we defined. This means that
the degree of~$\mathcal{G}(2,n)$ is equal to the~degree of~$M$.
We shall prove Theorem~\ref{thm:conj55} by showing that~$M$ has
degree~$2 C_{n-1}-1$. This will be done by constructing a
combinatorial model for the~simplicial complex  that is
represented by the~squarefree monomial ideal~$M$.

We begin with Young's poset on~$\binom{[n]}{2}$.
The elements in~this poset are the~variables~$\psi_{ij}$
with~$1 \leq i < j \leq n$, and the~order relation~$\preceq$ is defined by setting
~$\psi_{ij} \preceq \psi_{kl}$
if and only if~$i \leq k$ and~$j \leq l$.
The incomparable pairs are precisely the~leading monomials
$\psi_{il} \psi_{jk}$ for the~Grassmannian~${\rm Gr}(2,n)$.
The simplices in~the~associated
simplicial complex are the~chains in~the~poset. The
degree is the~number of maximal chains, which is
the Catalan number~$C_{n-2}$.

We define the~{\em outer chain} in~Young's poset to be the~following maximal chain:
$$ \psi_{12}  \preceq \psi_{13} \preceq  \cdots \preceq \psi_{1,n-1} 
\,\preceq \psi_{1n}\, \preceq \psi_{2n} \preceq \psi_{3n} \preceq \cdots \preceq \psi_{n-1,n}.~$$
The chain from~$\psi_{12}$ up to~$\psi_{1n}$ is called
the {\em lower outer chain}, and the~chain from~$\psi_{1n}$ up to~$\psi_{n-1,n}$ is
 the~{\em upper outer chain}. We now take another copy of
Young's poset, indexed by variables~$\xi_{ij}$, using artificial variables when $\{i,j\} \cap \{1,n\} = \varnothing$.
We link the $\xi$-poset together with the $\psi$-poset by the outer chains.
To be precise, we define a~poset~$\,{\rm P}_{2,n}$ on~$n(n-1)$ elements by
taking the~union of the~two Young's posets together with
the additional cover relations
~$\xi_{ij} \to \psi_{ij}$ whenever~$\{i,j\} \cap \{1,n\} \neq \varnothing$.

\begin{example}($n = 4$) 
    The~poset~${\rm P}_{2,4}$ has $12$ elements. Its Hasse diagram is shown in Figure~\ref{figure:psixiPoset}.
    The number of maximal chains of~${\rm P}_{2,4}$ equals
$9 = 2 \cdot 5 - 1$, which is the~CC degree of~${\rm Gr}(2,4)$.
\begin{figure}[h]
    \begin{center}
    \begin{tikzpicture}[scale=1]
    \draw[line width=15pt, red!30,opacity=0.2, rounded corners=15pt] (2.25,-0.25) -- (1,1) -- (-0.25,2.25)  -- cycle;
    \draw[line width=15pt, red!30,opacity=0.2, rounded corners=15pt] (2.25,0.75) -- (1,2) -- (-0.25,3.25)  -- cycle;
     \draw[line width=15pt, blue!30,opacity=0.2, rounded corners=15pt] (-0.25,1.75) -- (1,3) -- (2.25,4.25)  -- cycle;
    \draw[line width=15pt, blue!30,opacity=0.2, rounded corners=15pt] (-0.25,2.75) -- (1,4) -- (2.25,5.25)  -- cycle;
      \node (12) at (2,0) {$\xi_{12}$};
      \node (13) at (1,1) {$\xi_{13}$};
      \node (14) at (0,2) {$\xi_{14}$};
      \node (23) at (2,2) {$\xi_{23}$};
      \node (24) at (1,3) {$\xi_{24}$};
      \node (34) at (2,4) {$\xi_{34}$};
      \node (q12) at (2,1) {$\psi_{12}$};
      \node (q13) at (1,2) {$\psi_{13}$};
      \node (q14) at (0,3) {$\psi_{14}$};
      \node (q23) at (2,3) {$\psi_{23}$};
      \node (q24) at (1,4) {$\psi_{24}$};
      \node (q34) at (2,5) {$\psi_{34}$};
      \draw (12) -- (13) -- (14) -- (24) -- (34);
      \draw     (13) -- (23) -- (24);
      \draw (q12) -- (q13) -- (q14) -- (q24) -- (q34);
      \draw     (q13) -- (q23) -- (q24);
      \draw (12) -- (q12);
      \draw (13) -- (q13);
      \draw (14) -- (q14);
      \draw (24) -- (q24);
      \draw (34) -- (q34);
    \end{tikzpicture}    
    \end{center} \vspace{-0.15in}
    \caption{\label{figure:psixiPoset} 
    The diagram highlights
        the two outer chains in the poset ${\rm P}_{2,4}$: the lower outer chains are pink, 
    and the upper outer chains are purple. The~outer chain in~the~$\xi$-poset consists of variables in~our polynomial ring~$\CC[\xi,\psi]$.
    The~label~$\xi_{23}$ is \underbar{not} a~coordinate on~$\mathcal{G}(2,4)$.}
\end{figure}
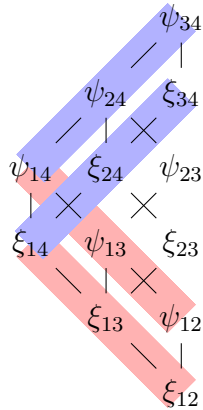
\end{example}

\begin{lemma} \label{lem:chains}
    The~number of maximal chains in~the~poset~${\rm P}_{2,n}$ equals~$2 \cdot C_{n-1}-1 = \frac{2}{n} \binom{2n-2}{n-1}-1$.    
\end{lemma}

\begin{proof}
   Every maximal chain in~${\rm P}_{2,n}$ transitions 
   precisely once from the~$\xi$-poset to the~$\psi$-poset.
   It therefore
      contains a~unique cover relation~$\xi_{ij} \to \psi_{ij}$
   where~$\{i,j\} \cap \{1,n\} \not= \varnothing$.
   We distinguish the~two cases when the~chain contains~$\xi_{in} \to \psi_{in}$ and when it contains~$\xi_{1j} \to \psi_{1j}$. 

    First consider maximal chains that move up to~$\psi$ on~the~upper outer chain.
    These are the chains containing a~cover relation~$\xi_{in} \to \psi_{in}$.  Such chains are in~bijection with maximal chains in~Young's poset~$\binom{[n+1]}{2}$, so their number is~$C_{n-1}$.
     This bijection is seen by replacing the~labels~$\psi_{in}$ with~$\xi_{i,n + 1}$ for~$i=1,2,\ldots,n-1$,
        and by adding a~new top element~$\xi_{n,n+1}$.
      
        Next we consider maximal chains in~${\rm P}_{2,n}$ that transition from~$\xi$ to~$\psi$
        on~the~lower outer chain. Such maximal chains contain a~unique cover relation~$\xi_{1j} \to \psi_{1j}$. These maximal chains correspond to
     maximal chains in~the~first case by reversing the~order of the~poset and switching~$\psi$ and~$\xi$ variables. 
     Hence there are also~$C_{n-1}$ maximal chains in~the~second class.

There is precisely one maximal chain that is covered by both cases.
This is the~chain which uses the~transition~$    \xi_{1n} \to \psi_{1n}$. 
This is the~concatenation of the~lower outer~$\xi$-chain
and the~upper outer~$\psi$-chain. This chain is counted twice, 
so we subtract one in~our total count.
We conclude that the~number of
    maximal chains in~the~poset~${\rm P}_{2,n}$ equals~$2 C_{n-1}-1$.
\end{proof}

\begin{proof}[Proof of Theorem~\ref{thm:conj55}]
We now change the~labels of the~poset~${\rm P}_{2,n}$
by replacing~$\xi_{ij}$ with~$\psi_{ij}$ whenever
$\{i,j\} \cap \{1,n\} =  \varnothing$. Since such pairs 
$\xi_{ij}$ and~$\psi_{ij}$ are incomparable, each maximal chain in~${\rm P}_{2,n}$ with the~new labeling 
consists of~$2n-2$ distinct variables in~our coordinate ring~$\CC[\xi,\psi]$.
These maximal chains are the~facets in~a~pure simplicial
complex of dimension~$2n-3$ on~$2n-3 + \binom{n}{2}$ vertices.
By Lemma~\ref{lem:chains}, the~number of facets is equal to~$2 C_{n-1}-1$.

Consider now the~ideal~$M$ which is generated
by the~underlined monomials in~Lemma~\ref{GrBasisG(2,n)}.
These squarefree quadratic monomials are
precisely the~incomparable pairs in~${\rm P}_{2,n}$
that have distinct labels after the~relabeling. Hence
$M$ is the~Stanley-Reisner ideal of our pure
simplicial complex. The~degree of~$M$
is the~number of facets, which is~$2 C_{n-1}-1$.
This is therefore the~degree of~$\mathcal{G}(2,n)$,
which equals  the~CC degree of~${\rm Gr}(2,n)$
by Theorem~\ref{thm:graphmap}.
\end{proof}

We have proved \cite[Conjecture 5.5]{FSS} by means of a
Gr\"obner degeneration. This replaces the irreducible variety~$\mathcal{G}(2,n)$
by a union of~$2 C_{n-1}-1$ coordinate subspaces in
$\PP^{2n-4} \times \PP^{\binom{n}{2}-1}$.
In what follows, we describe
a  toric degeneration that underlies
our Gr\"obner degeneration.
Recall that toric degenerations  \cite{LB, MohClarke, FFFM}
are described algebraically by
Khovanskii basis \cite{KM}.

We fix a monomial order on the polynomial ring
$\CC[s,t,{\bf x}]$ which selects the diagonal monomial~$x_{1i} x_{2j}$ to be the leading monomial for the each of the~$2 \times 2$ minors of
(\ref{eq:matrix2n}).
Let~$\mathcal{T}(2,n)$ denote the toric variety in
$\PP^{2n-4} \times \PP^{\binom{n}{2}-1}$ which is obtained by 
passing to these  leading monomials in (\ref{eq:ringhomo2a}).
Equivalently,
the toric ideal of~$\mathcal{T}(2,n)$ is the kernel of ring homomorphism
\begin{equation}
\label{eq:ringhomo2b} \begin{array}{rclccc}
\CC[\xi,\psi] & \rightarrow& \CC[s,t, \bf x],   &\xi_{1n}  \mapsto s, & \xi_{1i} \mapsto s \cdot x_{2i}, 
& \xi_{jn}  \mapsto s \cdot  x_{1j},  \\ 
\psi_{ij}  & \mapsto &  t \cdot x_{1i} x_{2j}, &\psi_{1n}  \mapsto t, &\psi_{1i}  \mapsto t \cdot x_{2i}, 
&\psi_{jn}  \mapsto t \cdot x_{1j},
\end{array} \quad
\text{for }1 < i < j < n. 
\end{equation}

\begin{proposition}\label{Thm:Khovanskii(2,n)}
The minors in (\ref{eq:ringhomo2a}) form a Khovanskii basis with respect to the diagonal monomial order.
Hence the toric variety~$\mathcal{T}(2,n)$ is a toric degeneration of the graph~$\mathcal{G}(2,n)$.
\end{proposition}

\begin{proof}
The~$2 \times 2$ minors of (\ref{eq:matrix2n}) form
a Khovanskii basis of the Grassmannian~${\rm Gr}(2,n)$;
see e.g.~\cite[Section 14.3]{CCA}.
The toric degeneration is given by the
Gel'fand-Tsetlin polytope ${\rm GT}(2,n)$.
Our variety~$\mathcal{T}(2,n)$ is the graph of the associated monomial parametrization.

By the same argument  as in  Lemma \ref{GrBasisG(2,n)},
the kernel of (\ref{eq:ringhomo2b}) has the following Gr\"obner basis:
$$ \begin{small} \begin{array}{clll}
 \binom{n}{4} & {\rm binomials} & 
    \underline{\psi_{il} \psi_{jk}}- \psi_{ik} \psi_{jl} 
    &   {\rm for} \,\, 1 \leq i \! < \!  j \! <\! k \! <\! l \leq n, \smallskip \\ 
 \binom{n-1}{2} & {\rm binomials} & 
   \underline{\psi_{in} \xi_{jn}} - \psi_{jn} \xi_{in} 
 & {\rm for} \,\, 1 \leq i < j < n, \smallskip \\
 (n-2)^2 & {\rm binomials} & 
 \underline{\psi_{1j}\xi_{kn}} -  \psi_{kn} \xi_{1j}& 
 {\rm for} \,\,1 <\, j\,, \, k \, < n, \smallskip \\
 \binom{n-1}{2} & {\rm binomials} & 
  \underline{\psi_{1k}\xi_{1l}} -  \psi_{1l}\xi_{1k} &
   {\rm for} \,\, 1 < k < l \leq n,  \smallskip \\
    \binom{n-2}{3} & {\rm binomials} & 
\underline{\xi_{1l} \psi_{jk}} - \xi_{1k} \psi_{jl}  &
{\rm for} \,\, 1 < j<k<l < n, \\
\binom{n-1}{3} & {\rm binomials} & 
\underline{\xi_{in} \psi_{jk}} - \xi_{jn} \psi_{ik} & 
{\rm for} \,\, 1 \leq i < j < k <n.
    \end{array} \end{small}
$$
Each of these binomials lifts to a relation on~$ \mathcal{G}(2,n)$,
by  adding one trailing term to those in the
first, fifth and sixth group.
This yields the Khovanskii basis property.
\end{proof}

The Gel'fand-Tsetlin polytope ${\rm GT}(2, n)$ is
the convex hull of the exponent vectors of the~$\binom{n}{2}$ images of~$\psi_{ij}$ in (\ref{eq:ringhomo2b}).
This is the {\em Newton-Okounkov body} \cite{LB} for  the toric degeneration of
 the Grassmannian $\mathrm{Gr}(2, n)$ with respect to the diagonal monomial order. 
We define the  {\em Cayley-Gel'fand-Tsetlin polytope} ${\rm CGT}(2,n)$ to be the
 Newton-Okounkov body of the toric degeneration~$\mathcal{T}(2,n)$.
 We here identify this toric variety
  with the projectivization of its cone in $\A^{2n-3} \times \A^{{n \choose 2}}$. 
Thus ${\rm CGT}(2,n)$ is the $(2n-3)$-dimensional
 polytope whose vertices are the exponent vectors of the $2n-3 + \binom{n}{2}$ monomials in (\ref{eq:ringhomo2b}). 
Here is a geometric construction:
 \begin{equation}
 \label{eq:CGT} {\rm CGT}(2, n) \,\,=\,\,
 \mathrm{conv}\bigl({\rm GT}(2, n) \!\times\! \{0\} \,\cup\, \Delta_{2n-4}\! \times\! \{1\}\bigr)
 \,\,\, \subset \,\,\RR^{\binom{n}{2}} \times \RR.
 \end{equation}
 In words, $ {\rm CGT}(2, n)$ is  the {\em Cayley sum} of ${\rm GT}(2,n)$ with the
 simplex $\Delta_{2n-4}$ that is formed by the $2n-3$
 coordinate points in $\RR^{\binom{n}{2}}$ which are indexed by pairs $i,j$ with $\{i,j\} \cap \{1,n\} \not= \varnothing$.

\begin{corollary} \label{eq:liftedGT}
The polytope ${\rm CGT}(2, n)$
has normalized volume~$\frac{2}{n} \binom{2n-2}{n-1} {-} 1$.
\end{corollary}

\begin{proof}
We use results from \cite[Chapter 8]{GBCP}.
We saw  that
the binomial generators of the toric ideal of~$\mathcal{T}(2,n)$
form a Gr\"obner basis with squarefree leading monomials.
Its initial ideal
 defines a regular unimodular
triangulation of the
Cayley-Gel'fand-Tsetlin polytope ${\rm CGT}(2,n)$
   into $\frac{2}{n} \binom{2n-2}{n-1} - 1$ simplices.
 These simplices are the maximal chains in the poset~${\rm P}_{2,n}$. 
\end{proof}

We conclude this section with a census of small 
Cayley-Gel'fand-Tsetlin polytopes.

\begin{example} For each polytope ${\rm CGT}(2,n)$
we list the $f$-vector and~Ehrhart polynomial.
For instance, ${\rm CGT}(2,6)$ is a $9$-dimensional polytope with
$24$ vertices and $15$ facets: \\

\noindent
$n=4$: ${\rm dim} = 5$, $\, (11,32,42,28,9), \,\, \frac{1}{5!} (t+1)(t+2)(t+3)({\bf 9}t^2 + 26t + 20) , $ \smallskip \\
$n=5$: ${\rm dim} = 7$, $(17,77,166,200,141,57,12), $\\
\phantom{dodo} \hfill $ 
\frac{1}{7!}(t+1)(t+2)(t+3)(t+4)({\bf 27}t^3+164t^2+313t+210),$ \smallskip \\
$n=6$: ${\rm dim} = 9$, $(24,152,467,836,941,685,321,93,15),$ \\
\phantom{dodo} \hfill $ 
\frac{1}{9!}(t+1)(t+2)(t+3)(t+4)(t+5)({\bf 83}t^4+861t^3+3172t^2+4956t+3024).$
\end{example}

\section{Khovanskii bases and Gr\"obner bases for $\mathcal{G}(d, n)$}
\label{sec4}

This section  concerns the extension of Proposition \ref{Thm:Khovanskii(2,n)}
from~$d=2$ to~$d \geq 3$.
The CC degree of~${\rm Gr}(d,n)$ should be
the normalized volume of a polytope from a nice  toric degeneration 
of the graph~$\mathcal{G}(d,n)$.
It turns out that the general case is considerably more
difficult than the~$d=2$ case which was treated in Section \ref{sec3}.
From now on we assume~$d \geq 3$. The construction that follows
is related to the Cayley-Gel'fand-Tsetlin polytope via \cite[Section~4.1]{ABM}.

There are many known
toric degenerations of the Grassmannian~${\rm Gr}(d,n)$.
Our aim is to describe one particular degeneration of~${\rm Gr}(d,n)$
which can be lifted to the graph~$\mathcal{G}(d,n)$.
This will give rise to the desired toric degeneration~$\mathcal{T}(d,n)$
and its associated polytope.

We now index the coordinates~$\psi_\sigma$ and~$\xi_\sigma$
using the convention in \cite[Example 2.11]{IgMal22}.
Each element in~$\binom{[n]}{d}$ is represented by its unique
 {\em PBW tuple}~$\sigma = (\sigma_1,\ldots,\sigma_d)$. This means: 
\begin{itemize}
\item the entries~$\sigma_i$
are distinct elements of~$\{1,2,\ldots,n\}$, \vspace{-0.1in}
\item~$\sigma_i \leq d$ implies~$\sigma_i = i$,  \vspace{-0.1in}
\item and~$d < \sigma_j <\sigma_k$ implies~$j < k$.
\end{itemize}
We fix a monomial order on~$\CC[s,t,{\bf x}]$ which selects  diagonal leading terms for  PBW tuples:
\begin{equation}
\label{eq:PBWorder}
{\rm in} \bigl(\psi_\sigma \bigr) \,= \, x_{1\sigma_1} x_{2\sigma_2} \,\cdots\, x_{d\sigma_d}. 
\end{equation}
It is known from \cite{FFFM, MAKHLIN2022105541, IgMal22} that 
the $d \times d$ minors of the $d \times n$ matrix $(x_{ij})$ form a~Khovanskii basis
for the Grassmannian~${\rm Gr}(d,n)$ with respect to this monomial order.
Here, we can either take all $dn$ matrix entries $x_{ij}$ to be unknowns, or we can 
use a birational parametrization as in the Introduction, where
$x_{ij}$ is a variable if~$j \geq d+1$, and~$x_{ij} \in \{0,1\}$ if~$j \leq d$.
The resulting toric degeneration has the {\em Hibi-Li ideal}
\cite{HibiLi, MAKHLIN2022105541}
 of
a distributive lattice on~$\binom{[n]}{d}$ which is known as the
{\em PBW poset}. 
The PBW poset for~$d=3$ and~$n=6$ is shown in~Figure~\ref{figure:PBW}.
\begin{figure}[h]
\begin{center}
\vspace{-0.1in}
    \begin{tikzpicture}[scale=0.88]
      \node (123) at (1,0) {$123$};
      \node (124) at (0,1) {$124$};
      \node (143) at (-1,2) {$143$};
      \node (125) at (1,2) {$125$};
      \node (423) at (-2,3) {$423$};
      \node (145) at (0,3) {$145$};
      \node (126) at (2,3) {$126$};
      \node (425) at (-1,4) {$425$};
      \node (153) at (0,4) {$153$};
      \node (146) at (1,4) {$146$};
      \node (453) at (-1,5) {$453$};
      \node (426) at (0,5) {$426$};
      \node (156) at (1,5) {$156$};
      \node (523) at (-2,6) {$523$};
      \node (456) at (0,6) {$456$};
      \node (163) at (2,6) {$163$};
      \node (526) at (-1,7) {$526$};
      \node (463) at (1,7) {$463$};
      \node (563) at (0,8) {$563$};
      \node (623) at (1,9) {$623$};
\pgfdeclarelayer{bg} 
  \pgfsetlayers{bg,main}
  \begin{pgfonlayer}{bg}
\draw [gray, fill = lightgray] (123) circle [radius = 3.5mm];
\draw [gray, fill = lightgray] (124) circle [radius = 3.5mm];
\draw [gray, fill = lightgray] (125) circle [radius = 3.5mm];
\draw [gray, fill = lightgray] (126) circle [radius = 3.5mm];
\draw [gray, fill = lightgray] (143) circle [radius = 3.5mm];
\draw [gray, fill = lightgray] (423) circle [radius = 3.5mm];
\draw [gray, fill = lightgray] (153) circle [radius = 3.5mm];
\draw [gray, fill = lightgray] (523) circle [radius = 3.5mm];
\draw [gray, fill = lightgray] (163) circle [radius = 3.5mm];
\draw [gray, fill = lightgray] (623) circle [radius = 3.5mm];
\end{pgfonlayer}
  \begin{pgfonlayer}{bg}
\draw (123) -- (124) -- (125) -- (126) -- (146) -- (156) -- (163) --(463) -- (563) -- (526) -- (456) -- (156) -- (153);
\draw (426) -- (425) -- (453) -- (456) -- (463);
\draw (623) -- (563);
\draw (526) -- (523) -- (453) -- (153) -- (145) -- (125);
\draw (426) -- (146) -- (145) -- (425) -- (423) -- (143) -- (124);
\draw (145) -- (143);
\draw (426) -- (456);
\end{pgfonlayer}
    \end{tikzpicture}    
      \end{center}
      \vspace{-0.15in}
      \caption{\label{figure:PBW} 
      This diagram shows the PBW poset for $d=3,n=6$.
            The~$20$ PBW tuples~$\sigma$ index the 
coordinates~$\psi_\sigma$.
The~$10$ PBW tuples that also index coordinates $\xi_\sigma$~are marked in grey.
The incomparable pairs in the poset generate our initial monomial ideal of~${\rm Gr}(3,6)$.}    
\end{figure}
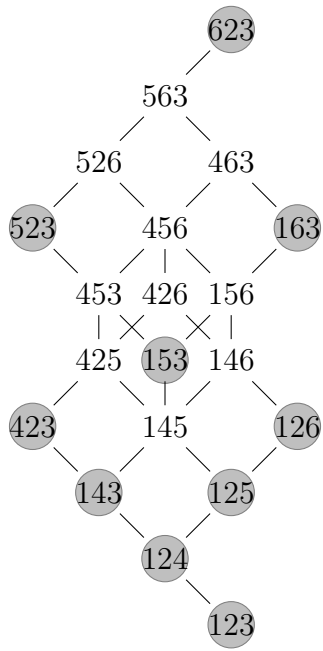

The Hibi-Li ideal describes the PBW degeneration of~${\rm Gr}(d, n)$.
In representation theory (cf.~\cite{ABM, FFFM}), this is   known as the
 \emph{Feigin–Fourier–Littelmann–Vinberg (FFLV) degeneration},
 and  we write  ${\rm FFLV}(d,n)$ for the associated polytope
 of dimension $d(n-d)$ with $\binom{n}{d}$ vertices.
 Thus, ${\rm FFLV}(d,n)$ is the convex hull of the exponent vectors
 of the  $\binom{n}{d}$ monomials in (\ref{eq:PBWorder}).
 
 In analogy to (\ref{eq:CGT}), we now define the
 \emph{Cayley–Feigin–Fourier–Littelmann–Vinberg polytope}~${\rm CFFLV}(d, n)$. This is a polytope of dimension
$d(n-d)+1$ with $\binom{n}{d}+d(n-d)+1$ vertices.
Geometrically, we construct ${\rm CFFLV}(d, n)$ as
the Cayley sum in $\RR^{\binom{n}{d}} \times \RR$ of the polytope ${\rm FFLV}(d,n)$ with the 
 simplex $\Delta_{d(n-d)}$ that is formed by the $d(n-d)+1$
 coordinate points which are indexed by 
 PBW tuples $\sigma$ with at most one entry $\sigma_j$ in~$[n] \backslash [d] = \{d+1,\ldots,n\}$.

 We conjecture that ${\rm CFFLV}(d,n)$ represents a toric degeneration $\mathcal{T}(d,n)$ of the
 projectivization of the affine cone in $\A^{d(n-d)+1}\times \A^{{n\choose d}}$ over the graph $\mathcal{G}(d,n)$. Thus $\mathcal{T}(d,n)$
 is a toric variety in~$\,\PP^{d(n-d)+\binom{n}{d}}$, and its
 degree equals the CC degree of~${\rm Gr}(d,n)$.
 Furthermore, the polytope
 ${\rm CFFLV}(d,n)$ has a  regular unimodular triangulation.
The number of simplices is the CC degree of ${\rm Gr}(d,n)$.
This is the content of
Conjecture~\ref{thm:KhovanskiiPBW}.
We illustrate all concepts and results for  $d=3 , n=6$ with a 
computation in~{\tt Macaulay2}~\cite{M2}.

\begin{example}[$d=3,n=6$] \label{ex:36}
The following code uses~${\tt p}$ and~${\tt q}$ for the Greek letters~$\psi$ and~$\xi$:
\begin{footnotesize}
\begin{verbatim}
R = QQ[t,s,x14,x15,x16,x24,x25,x26,x34,x35,x36,
       p623,q623,p563,p463,p526,p163,q163,p456,p523,q523,
       p156,p426,p453,p146,p153,q153,p425,p126,q126,p145,
       p423,q423,p125,q125,p143,q143,p124,q124,p123,q123, MonomialOrder=>Eliminate 11];
I = ideal(p123-t, q123-s,
          p145-t*(x24*x35-x25*x34), p146-t*(x24*x36-x26*x34), p156-t*(x25*x36-x26*x35),
          p425-t*(x14*x35-x15*x34), p426-t*(x14*x36-x16*x34), p453-t*(x14*x25-x15*x24),
          p463-t*(x14*x26-x16*x24), p526-t*(x15*x36-x16*x35), p563-t*(x15*x26-x16*x25),
          p456-t*(x14*x25*x36-x14*x26*x35-x15*x24*x36+x15*x26*x34+x16*x24*x35-x16*x25*x34),
          p124-t*x34, p125-t*x35, p126-t*x36, p143-t*x24, p153-t*x25, p163-t*x26,
          p423-t*x14, p523-t*x15, p623-t*x16, q124-s*x34, q125-s*x35, q126-s*x36,
          q143-s*x24, q153-s*x25, q163-s*x26, q423-s*x14, q523-s*x15, q623-s*x16);
G = ideal selectInSubring(1,gens gb I); toString mingens G
codim G, degree G, betti mingens G
M = ideal leadTerm ideal selectInSubring(1,gens gb I);
codim M, degree M, betti mingens M
\end{verbatim}
\end{footnotesize}
The output~${\tt G}$ is the homogeneous prime ideal of~$\mathcal{G}(3,6)$. It is minimally generated
by~$106$ quadrics. 
The command {\tt degree G} shows that $\mathcal{G}(3,6)$ has degree $250$.
This coincides with  the CC degree of~${\rm Gr}(3,6)$, by
 \cite[Example 5.7]{FSS}.
The output ${\tt M}$ is the initial ideal for the
reverse lexicographic monomial order given by a
linear extension of the PBW poset in Figure~\ref{figure:PBW}.

We compute the toric degeneration $\mathcal{T}(3,6)$ with the same code after replacing
 each parenthesized minor  by its leading monomial, e.g.~replace 
 {\tt (x24*x35-x25*x34}) by   {\tt x24*x35}. The output is a toric ideal of degree $250$, minimally generated by 
 $106$ quadrics and 
 one cubic.  The Khovanskii basis property holds because
 each of these binomials lifts to polynomial in~${\tt G}$.
 
    This toric ideal represents the
$10$-dimensional polytope ${\rm CFFLV}(3, 6)$.
 Its $f$-vector equals
$$ f_{{\rm CFFLV}(3, 6)} \,=\,(30, 236, 901, 2017, 2873, 2695, 1672, 670, 163, 21).$$
The command  {\tt hilbertPolynomial G} computes the Ehrhart polynomial
 of this polytope.
The graph $\mathcal{G}(3,6)$ and its toric degeneration $\mathcal{T}(3,6)$ have
the same initial monomial ideal, here denoted by {\tt M}.
This is generated by
$ 132$ squarefree polynomials: 106 quadrics, 21 cubics, and 5 quartics. 
We view  {\tt M} as the Stanley-Reisner ideal of a pure
simplicial complex of dimension $10$ with $250$ maximal simplices.
This  is our  unimodular triangulation of ${\rm CFFLV}(3, 6)$.
\end{example}

We conjecture the following result, which should
hold for arbitrary $d$ and $n$.

\begin{conjecture}  \label{thm:KhovanskiiPBW}
We fix the monomial orders on $\CC[\xi,\psi]$ and $\CC[s,t,{\bf x}]$ described above.
\begin{enumerate}
\item The FFLV Khovanskii basis of the Grassmannian
${\rm Gr}(d,n)$ 
  lifts to the graph~$\mathcal{G}(d,n)$.
\item The ideals of~$\mathcal{G}(d,n)$ and~$\mathcal{T}(d,n)$ have squarefree Gr\"obner bases.  
\item The normalized volume of the
Cayley–Feigin–Fourier–Littelmann–Vinberg polytope
${\rm CFFLV}(d, n)$ is equal to 
the coupled cluster degree of the Grassmannian ${\rm Gr}(d,n)$.
\end{enumerate}
\end{conjecture}

A first version of this article was posted in October 2023.
A follow-up paper by Evgeny Feigin \cite{Fei} from March 2024
establishes a toric degeneration of graph~$\mathcal{G}(d,n)$
for arbitrary $d$ and $n$. Feigin's work rests on a different method
and it does not prove Conjecture \ref{thm:KhovanskiiPBW}.
We are hopeful that, by combining the two approaches, an explicit 
Khovanskii basis for $\mathcal{G}(d,n)$ can be found.
We conclude the present paper with two remarks of a computational nature.

\begin{remark} \label{rmk:symbolic}
The computation in Example \ref{ex:36} 
serves as a proof for Conjecture \ref{thm:KhovanskiiPBW}
in the case $(d,n)=(3,6)$.
The same computation in {\tt Macaulay2} terminated successfully for
$\, (d,n) \, = \,(3,7),(3,8),(3,9),(4,8)$.
We therefore have a complete computational proof of 
 Conjecture~\ref{thm:KhovanskiiPBW} in these special cases.
 The {\tt Macaulay2} code for these computations can be found on
{\tt MathRepo} at the URL in (\ref{eq:mathrepo}).
For the cases $(d,n) = (3,10), \,(4,9)$, we computed the
degree of $\mathcal{T}(d,n)$ in {\tt Macaulay2}, so the
correctness of the final number assumes the validity of
Conjecture \ref{thm:KhovanskiiPBW} (1).
   The  CC degrees for the three largest among the above cases was reported in Example~\ref{ex:bigCC}.
 \end{remark}

\begin{remark}
We solved the CC equations (\ref{eq:CCeqns}) using the numerical algebraic geometry software
{\tt HomotopyContinuation.jl}~\cite{HomotopyContinuation} for the cases in Remark~\ref{rmk:symbolic}.
This rests on the methodology described in \cite[Section 6]{FSS}.
This numerical computation
yields independent support for the fact that the CC degree of the Grassmannian equals
the  volume of the CFFLV polytope.
\end{remark}

 Khovanskii bases
have emerged as an important tool in the recent literature on
polynomial system solving \cite{BPT} and its applications in the
physical sciences \cite{BBPMZ}.
Theorems \ref{thm:graphmap} and~\ref{thm:conj55} and Conjecture \ref{thm:KhovanskiiPBW}
lay the foundation for a new class of
custom-taylored homotopies for numerically solving the CC equations.
Viewed from this angle, this article
offers the potential for considerable practical implications in coupled cluster theory,
well beyond those in \cite{FO, FSS}.

\section*{Acknowledgements.} VB was supported by the Deutsche Forschungsgemeinschaft (DFG), Projekt-Nr.~445466444.

\bibliographystyle{plain}

\bigskip
 \bigskip

\noindent
\footnotesize
{\bf Authors' addresses:}

\smallskip

\noindent Viktoriia Borovik,
Universit\"at Osnabr\"uck
\hfill {\tt vborovik@uni-osnabrueck.de}

\noindent Bernd Sturmfels,
MPI-MiS Leipzig
\hfill {\tt bernd@mis.mpg.de}

\noindent Svala Sverrisd\'ottir,
UC Berkeley
\hfill {\tt svalasverris@berkeley.edu}

\end{document}